\documentclass{amsart}
\usepackage{graphicx}
\usepackage{amssymb}
\usepackage{amsfonts}
\swapnumbers
\newtheorem{theorem}{Theorem}[section]
\newtheorem{lemma}[theorem]{Lemma}
\newtheorem{corollary}[theorem]{Corollary}

\theoremstyle{definition}

\newtheorem{example}[theorem]{Example}

\numberwithin{equation}{section}
 \theoremstyle{plain}
 
 \numberwithin{equation}{section} 
 \numberwithin{figure}{section} 
 \theoremstyle{plain}
 \theoremstyle{remark}
 \newtheorem*{acknowledgement*}{Acknowledgement}

\newcommand{\cA}{{\mathcal A}}

\newcommand{\cF}{{\mathcal F}}
\newcommand{\cG}{{\mathcal G}}
\newcommand{\cH}{{\mathcal H}}

\newcommand{\cL}{{\mathcal L}}

\newcommand{\cN}{{\mathcal N}}

\newcommand{\cX}{{\mathcal X}}

\newcommand{\Om}{{\Omega}}
\newcommand{\om}{{\omega}}
\newcommand{\ve}{{\varepsilon}}
\newcommand{\del}{{\delta}}
\newcommand{\Del}{{\Delta}}

\newcommand{\Gam}{{\Gamma}}

\newcommand{\Sig}{{\Sigma}}
\newcommand{\sig}{{\sigma}}
\newcommand{\al}{{\alpha}}
\newcommand{\be}{{\beta}}
\newcommand{\ka}{{\kappa}}
\newcommand{\la}{{\lambda}}

\newcommand{\bbN}{{\mathbb N}}
\newcommand{\bbP}{{\mathbb P}}
\newcommand{\bbR}{{\mathbb R}}

\newcommand{\bbI}{{\mathbb I}}

\newcommand{\bfY}{{\bf Y}}
\newcommand{\bfX}{{\bf X}}

\begin{document}
\title[]{Geometric law for multiple returns\\
until a hazard}%
 \vskip 0.1cm
 \author{ Yuri Kifer and Ariel Rapaport\\
\vskip 0.1cm
 Institute  of Mathematics\\
Hebrew University\\
Jerusalem, Israel}%
\address{
Institute of Mathematics, The Hebrew University, Jerusalem 91904, Israel}
\email{ kifer@math.huji.ac.il, ariel.rapaport@mail.huji.ac.il}%

\thanks{ }
\dedicatory{  }
 \date{\today}
\begin{abstract}\noindent
For a $\psi$-mixing stationary process $\xi_0,\xi_1,\xi_2,...$ we consider
the number $\cN_N$ of multiple recurrencies $\{\xi_{q_i(n)}\in\Gam_N,\,
i=1,...,\ell\}$ to a set $\Gam_N$ for $n$ until the moment $\tau_N$ (which
we call a hazard) when another multiple recurrence $\{\xi_{q_i(n)}\in\Del_N,\,
i=1,...,\ell\}$ takes place for the first time where $\Gam_N\cap\Del_N=
\emptyset$ and $q_i(n)<q_{i+1}(n),\, i=1,...,\ell$ are nonnegative increasing
functions taking on integer values on integers. It turns out that if
$P\{\xi_0\in\Gam_N\}$ and $P\{\xi_0\in\Del_N\}$ decay in $N$ with the same
speed then $\cN_N$ converges weakly to a geometrically distributed random
variable. We obtain also a similar result in the dynamical systems setup
considering a $\psi$-mixing shift $T$ on a sequence space $\Om$ and study
the number of multiple recurrencies $\{ T^{q_i(n)}\om\in A_n^b,\,
i=1,...,\ell\}$ until the first occurence of another multiple recurrence
$\{ T^{q_i(n)}\om\in A_m^a,\, i=1,...,\ell\}$ where $A_m^a,\, A_n^b$ are
cylinder sets of length $m$ and $n$ constructed by sequences $a,b\in\Om$,
 respectively, and chosen so that their probabilities have the same order.
This work is motivated by a number of papers on asymptotics of numbers
of single and multiple returns to shrinking sets, as well as by the papers
on open systems studying their behavior until an exit through a "hole".
\end{abstract}
\maketitle
\markboth{Yu.Kifer and A. Rapaport}{Multiple returns until a hazard}
\renewcommand{\theequation}{\arabic{section}.\arabic{equation}}
\pagenumbering{arabic}

\section{Introduction}\label{sec1}\setcounter{equation}{0}

Let $\xi_0,\,\xi_1,\,\xi_2,...$ be a sequence of independent identically
distributed (i.i.d.) random variables and $\Gam_N,\,\Del_N$ be a sequence
of sets such that $\Gam_N\cap\Del_N=\emptyset$ and
\[
\lim_{N\to\infty}\la^{-1}NP\{\xi_0\in\Gam_N\}=\lim_{N\to\infty}\nu^{-1}N
P\{\xi_0\in\Del_N\}=1.
\]
Then by the classical Poisson limit theorem
\[
S_N^{(\la)}=\sum_{n=0}^{N-1}\bbI_{\Gam_N}(\xi_n)\,\,\mbox{and}\,\,
S_N^{(\nu)}=\sum_{n=0}^{N-1}\bbI_{\Del_N}(\xi_n),
\]
where $\bbI_\Gam$ is the indicator of a set $\Gam$,
converge in distribution to Poisson random variables with parameters $\la$
and $\nu$, respectively. On the other hand, if
\[
\tau_N=\min\{ n\geq 0:\,\xi_n\in\Gam_N\}
\]
then it turns out that the sum
\[
S_{\tau_N}=\sum_{n=0}^{\tau_N-1}\bbI_{\Del_N}(\xi_n),
\]
which counts returns to $\Del_N$ until arriving at $\Gam_N$,
converges in distribution to a geometric random variable $\zeta$ with the
parameter $p=\la(\la+\nu)^{-1}$, i.e. $P\{\zeta=k\}=(1-p)^kp$.

Next, consider a more general setup which includes increasing functions
$q_i(n),\, i=1,...,\ell,\, n\geq 0$ taking on integer values on integers
and satifying $0\leq q_1(n)<q_2(n)<\cdots q_\ell(n)$ with all differences
$q_i(n)-q_{i-1}(n)$ tending to $\infty$ as $n\to\infty$. Here
we deal with "nonconventional" sums
\begin{equation}\label{1.1}
S_{\tau_N}=\sum_{n=0}^{\tau_N-1}\prod_{i=1}^\ell\bbI_{\Del_N}(\xi_{q_i(n)})
\end{equation}
defining now
\begin{equation}\label{1.2}
\tau_N=\min\{ n\geq 0:\,\prod_{i=1}^\ell\bbI_{\Gam_N}(\xi_{q_i(n)})=1\}
\end{equation}
and setting $\tau_N=\infty$ if the set in braces is empty. Now $S_{\tau_N}$
equals the number $\cN_N$ of multiple returns to $\Del_N$ until the first
multiple return to $\Gam_N$. It turns out that if
\begin{equation}\label{1.3}
\lim_{N\to\infty}\la^{-1}N(P\{\xi_0\in\Gam_N\})^\ell=\lim_{N\to\infty}\nu^{-1}N
(P\{\xi_0\in\Del_N\})^\ell=1
\end{equation}
then, again, $S_{\tau_N}$ converges in distribution to a geometric random
variable with the parameter $p=\la(\la+\nu)^{-1}$.

In the most general case in this setup we consider $\xi_0,\,\xi_1,\,\xi_2,...$
forming a $\psi$-mixing (see Section \ref{sec2}) stationary sequence of random
variables with $S_{\tau_N}$ defined again by (\ref{1.1}). We will show that if
(\ref{1.3}) holds true then, as in the i.i.d. case, $S_{\tau_N}$ will converge
in distribution to a geometric random variable with the parameter $p=
\la(\la+\nu)^{-1}$. In fact, we will obtain estimates for the total variation
distance between the distribution of $S_{\tau_N}$ and the geometric
distribution with the parameter $p=\la(\la+\nu)^{-1}$. On the other hand,
if the second equality in (\ref{1.3}) holds true and $S_N$ is the sum in
(\ref{1.1}) taken up to $N-1$ instead of $\tau_N-1$ then the distribution
of $S_N$ converges in total variation to the Poisson distribution with the
parameter $\nu$.

When $\ell=1$ the sum $S_{\tau_N}$ describes the number of returns to $\Del_N$
 by the sequence $\{\xi_n\}$ before reaching $\Gam_N$ which can be interpreted
as a "hole" through which the system (particle) exits and the count stops.
When $\ell>1$ the sum $S_{\tau_N}$ describes the number of multiple returns
to $\Del_N$ taking place at the moments $q_i(n),\, i=1,...,\ell$ until the
 system performs first multiple return to another set $\Gam_N$ (disjoint
 with $\Gam_N$) which we designate as a "hazard".

 We consider in this paper also another setup which comes from dynamical
 systems but has a perfect probabilistic sense, as well. Let $\zeta_k,\,
 k=0,1,2,...$ be a $\psi$-mixing discrete time process evolving on a finite
 or countable state space $\cA$. For each sequence $a=(a_0,a_1,a_2,...)\in
 \cA^\bbN$ of elements from $\cA$ and any $m\in\bbN$ denote by $a^{(m)}$
 the string $a_0,a_1,...,a_{m-1}$ which determines also an $m$-cylinder
 set $A^a_m$ in $\cA^\bbN$ consisting of sequences whose initial $m$-string
 coincides with $a_0,a_1,...,a_{m-1}$. Let $\tau^a_m$ be the first $l$ such
  that starting at the times
 $q_1(l),\, q_2(l),...,q_\ell(l)$ the process $\zeta_k=\zeta_k(\om),\,
 k\geq 0$ repeats the string $a^{(m)}=(a_0,...,a_{m-1})$. 
 Let $b=(b_0,b_1,...)\in\cA^\bbN,\, b\ne a$. We are interested
 in the number of $j<\tau^a_m$ such that process $\zeta_k$ repeats the string
  $b^{(n)}=(b_0,...,b_{n-1})$ starting at the
 times $q_1(j),\, q_2(j),...,q_\ell(j)$. Employing the left shift
 transformation $T$ on the sequence space $\cA^\bbN$ we can represent the
 number in question as a random variable on $\Om=\cA^\bbN$ given by the sum
 \begin{equation}\label{1.4}
 \Sig_{n,m}^{b,a}(\om)=\sum_{j=0}^{\tau_m^a-1}\prod_{i=1}^\ell\bbI_{A^b_n}
 (T^{q_i(j)}\om).
 \end{equation}
 We will show that for any $T$-invariant $\psi$-mixing probability measure $P$
 on $\Om$ and $P$-almost all $a,b\in\Om$ the distribution of random variables
 $\Sig^{b,a}_{n,m}$ approaches in the total variance distance as $n\to\infty$ 
 the geometric distribution with the parameter
 \[
 (P(A^a_m))^\ell\big ( (P(A^a_m))^\ell+(P(A^b_n))^\ell\big )^{-1}
 \]
 provided the ratio $P(A^b_n)/P(A^a_m)$ stays bounded away from zero and 
 infinity. In particular, if this ratio tends to $\la$ when $m=m(n)$ and
 $n\to\infty$ then the distribution of $\Sig_{n,m}^{b,a}$ converges
 in total variation distance to the geometric distribution with the 
 parameter $(1+\la^\ell)^{-1}$.

 Our results are applicable to larger  classes of dynamical systems and
 not only to shifts. Among such systems are smooth expanding endomorphisms
 of compact manifolds and Axiom A (in particular, Anosov)
 diffeomorphisms which have symbolic representations via Markov partitions
 (see \cite{Bo}). Then, in place of cylinder sets we can count multiple
 returns to an element of a Markov partition until first multiple return
 to another element of this partition. If for such dynamical systems we
 consider Sinai-Ruelle-Bowen type measures then the results can be extended
 to returns to geometric balls in place of elements of Markov partitions using
 approximations of the former by unions of the latter (see, for instance,
 the proof of Theorem 3 in \cite{HP}).
 The results remain true for some
 systems having symbolic representations with infinite alphabet, for instance,
 for the Gauss map $Tx=\frac 1x$ (mod 1), \, $x\in(0,1],\, T0=0$ of the unit
 interval considered with the Gauss measure $G(\Gam)=\frac 1{\ln 2}\int_\Gam
 \frac {dx}{1+x}$ which is known to be $T$-invariant and $\psi$-mixing with
 an exponential speed (\cite{He}). It seems that our geometric distribution
 results are new even for single return cases, i.e. when $\ell=1$.

 The motivation for the present paper is two-fold. On one hand, it comes from
  the series of papers deriving Poisson type asymptotics for distributions of
  numbers of single and multiple returns to appropriately shrinking sets
  (see, for instance, \cite{AV}, \cite{AS}, \cite{KR} and references there).
  On the other hand, our motivation was influenced by works on open dynamical
  systems which study dynamics of such systems until they exit the phase space
  through a "hole" (see, for instance, \cite{DWY} and references there). In
  our setup the number of multiple returns is studied until a "hazard" which
  is interpreted as certain $\ell$-tuple visits to a set which can be also
  viewed as a "hole". Then we can think on a system as a cluster of $\ell$
  particles which move together and loose one particle upon visiting a "hole"
  at prescribed times.

  The structure of this paper is as follows. In the next section we will
   describe precisely our setups and formulate main results. In Section
   \ref{sec3} we will prove our geometric limit theorem for the case of
   stationary processes and in Section \ref{sec4} this result will be
   derived for shifts.

 \section{Preliminaries and main results}\label{sec2}\setcounter{equation}{0}
\subsection{Stationary processes}\label{subsec2.1}
Our first setup includes a stationary sequence of random variables $\xi_0,\,
\xi_1,\,\xi_2,...$ defined on a complete probability space $(\Om,\cF,P)$
and a two parameter family of $\sig$-algebras $\cF_{mn}=\sig(\xi_m,
\xi_{m+1},...,\xi_n),\, m\leq n$, i.e. $\cF_{mn}$ is the minimal $\sig$-algebra
for which $\xi_m,\xi_{m+1},...,\xi_n$ are measurable. Recall, that the
$\psi$-dependence (mixing) coefficient between two $\sig$-algebras $\cG$ and
$\cH$ can be written in the form (see \cite{Br}),
\begin{eqnarray}\label{2.1}
&\psi(\cG,\cH)=\sup_{\Gam\in G,\Del\in\cH}\big\{\big\vert\frac {P(\Gam\cap\Del)}
{P(\Gam)P(\Del)}-1\big\vert,\, P(\Gam)P(\Del)\ne 0\big\}\\
&=\sup\{\| E(g|\cG)-E(g)\|_{L^\infty}:\, g\,\,\mbox{is}\,\,
\cH-\mbox{measurable and}\,\, E|g|\leq 1\}.\nonumber
\end{eqnarray}
Set also
\begin{equation}\label{2.2}
\psi(n)=\sup_{m\geq 0}\psi(\cF_{0,m},\cF_{m+n,\infty}).
\end{equation}
The sequence (process) $\xi_1,\xi_2,...$ is called $\psi$-mixing if
$\psi(1)<\infty$ and $\psi(n)\to 0$ as $n\to\infty$.

Our multiple recurrence setup includes also strictly increasing
functions $q_i,\, i=1,...,\ell$ taking on integer values on integers  and
satisfying
\begin{eqnarray}\label{2.3}
&0\leq q_1(n)<q_2(n)<...<q_\ell(n)\,\,\,\mbox{for all}\,\, n\geq 0\\
&\mbox{and}\,\,\, q(n)=\min_{k\geq n}\min_{1\leq i\leq\ell-1}( q_{i+1}(k)
-q_i(k))\to\infty\,\,\mbox{as}\,\, n\to\infty.\nonumber
\end{eqnarray}
 Set
\[
X_{n,\al}=\prod_{i=1}^\ell\bbI_{\Gam_\al}(\xi_{q_i(n)}),\,\,\al=0,1
\]
where $\Gam_0$ and $\Gam_1$ are disjoint Borel sets.
The sum $S_M=\sum_{n=0}^{M-1}X_{n,1}$ counts the number of multiple returns
of the sequence $\xi_1,\xi_2,...$ to $\Gam_1$ at times $q_1(n),\,
q_2(n),\,...,q_\ell(n)$ for $0\leq n< M$. Statistical properties
of such sums were studied in \cite{Ki} and \cite{KR}. Set
\begin{equation}\label{2.4}
\tau=\min\{ n\geq 0:\, X_{n,0}=1\}
\end{equation}
writing $\tau=\infty$ if the set in braces above is empty.
We will describe below the statistical properties of sums $S_{\tau}$
(setting $S_{0}=0$) which count the number of multiple returns of the sequence
 $\xi_1,\xi_2,...$
to $\Gam_1$ at times $q_1(n),\, q_2(n),\,...,q_\ell(n)$ until the random time
$\tau$ which we call a hazard.

For any two random variables or random vectors $Y$ and $Z$ of the same
dimension denote by $\cL(Y)$ and $\cL(Z)$ their distribution and by
\[
d_{TV}(\cL(Y),\,\cL(Z))=\sup_G|\cL(Y)(G)-\cL(Z)(G)|
\]
the total variation distance between $\cL(Y)$ and $\cL(Z)$ where the supremum
is taken over all Borel sets. We denote also by Geo$(\rho),\,\rho\in(0,1)$
the geometric distribution with the parameter $\rho$, i.e.
\[
\mbox{Geo}(\rho)\{ k\}=\rho(1-\rho)^k\,\,\mbox{for each}\,\, k\in\bbN=\{ 0,1,
...\}.
\]
Denote by $Q$ the distribution of $\xi_0$, i.e. $P\{\xi_0\in\Gam\}=Q(\Gam)$
for any Borel $\Gam\subset\bbR$.

\begin{theorem}\label{thm2.1} Let $\xi_0,\,\xi_1,\,\xi_2,...$ be a
$\psi$-mixing stationary process and assume that
the condition (\ref{2.3}) holds true.
Then for any disjoint Borel sets $\Gam_0,\,\Gam_1$ with $Q(\Gam_\al)>0,\,
\al=0,1$ and positive integers $M,\, R$ with $\psi(R)<2^{\frac 1{\ell+1}}-1$
we have
\begin{eqnarray}\label{2.5}
&d_{TV}(\cL(S_{\tau}),\,\mbox{Geo}(\rho))\leq C\bigg((1-Q(\Gam_0)^\ell)^M+
(Q(\Gam_0)^\ell+Q(\Gam_1)^\ell)\\
&\times\big(Q(\Gam_0)+Q(\Gam_1))MR+M\psi(R)+\sum_{n=0}^M\psi(q(n))\big)\bigg)
+2Q(\Gam_1)^\ell
\nonumber\end{eqnarray}
where $\rho=\frac {Q(\Gam_0)^\ell}{Q(\Gam_0)^\ell+Q(\Gam_1)^\ell}$
and the constant $C>0$ does not depend on $Q(\Gam_0),\, Q(\Gam_1)$, $M$ and $R$.
\end{theorem}

Next, let $\Gam_N,\,\Del_N,\, N=1,2,...$ be a sequence of pairs of disjoint
sets such that
\begin{equation}\label{2.6}
Q(\Gam_N),\, Q(\Del_N)\to 0\,\,\mbox{as}\,\, N\to\infty\,\,\mbox{and}\,\,
0<C^{-1}\leq\frac {Q(\Gam_N)}{Q(\Del_N)}\leq C<\infty
\end{equation}
for some constant $C$. Set
\begin{eqnarray*}
&X^{(N)}_{n,0}=\prod_{i=1}^\ell\bbI_{\Gam_N}(\xi_{q_i(n)}),\,
X^{(N)}_{n,1}=\prod_{i=1}^\ell\bbI_{\Del_N}(\xi_{q_i(n)})\\
&\tau_N=\min\{ n\geq 0:\, X^{(N)}_{n,0}=1\}\,\,\,\mbox{and}\,\,\,
S^{(N)}_M=\sum_{n=0}^{M-1}X^{(N)}_{n,1}.
\end{eqnarray*}

\begin{corollary}\label{cor2.2}
Suppose that the conditions of Theorem \ref{thm2.1} concerning the stationary
process $\xi_0,\xi_1,\xi_2,...$ and the functions $q_i(n),\, i=1,...,\ell$ are
satisfied. Let $\Gam_N,\,\Del_N,\, N=1,2,...$ be Borel sets satisfying
(\ref{2.6}). Then
\begin{equation}\label{2.7}
d_{TV}(\cL(S^{(N)}_{\tau_N}),\,\mbox{Geo}(\rho_N))\to 0\,\,\mbox{as}\,\,
N\to\infty
\end{equation}
where $\rho_N=Q(\Gam_N)^\ell(Q(\Del_N)^\ell+Q(\Gam_N)^\ell)^{-1}$. In
particular, if
\begin{equation}\label{2.8}
\lim_{N\to\infty}\frac {Q(\Del_N)}{Q(\Gam_N)}=\la
\end{equation}
then the distribution of $S^{(N)}_{\tau_N}$ converges in total variation
as $N\to\infty$ to the geometric distribution with the parameter
$(1+\la^\ell)^{-1}$.
\end{corollary}

The arguments in the proof of Theorem \ref{thm2.1} and Corollary \ref{cor2.2}
will yield also the following multiple recurrence results which generalize
some of the results from \cite{Ki} (where only independent and Markov sequences
$\xi_n,\, n\geq 0$ were considered).

\begin{theorem}\label{thm2.3} Let the conditions of Theorem \ref{thm2.1}
concerning the stationary process $\xi_0,\xi_1,\xi_2,...$ and the functions
 $q_i(n),\, i=1,...,\ell$ hold true. Let $\Gam$ be a Borel set,
 $X_n=\prod_{i=1}^\ell\bbI_\Gam(\xi_{q_i(n)})$ and $S_N=\sum_{n=0}^{N-1}X_n$.
  Then
 \begin{equation}\label{2.9}
 d_{TV}(\cL(S_N),\,\mbox{Pois}(\la))\leq CNQ(\Gam)^\ell\big( RQ(\Gam)+\psi(R))
 +\wp(CQ(\Gam)^\ell\sum_{n=0}^N\psi(q(n))\big)
 \end{equation}
 where $\la=NQ(\Gam)^\ell$, $\wp(x)=xe^{-x}$, $R<N$ is an arbitrary positive
 integer with $\psi(R)<2^{\frac 1{\ell+1}}-1$, $C>0$ is a constant which
 does not depend
 on $Q(\Gam)$, $N$ and $R$ and Pois$(\la)$ denotes the Poisson distribution
 with the parameter $\la$.
 \end{theorem}

 \begin{corollary}\label{cor2.4} Under the conditions of Theorem \ref{thm2.3}
 suppose that in place of one set $\Gam$ we have a sequence of Borel sets
 $\Gam_N$ such that
 \begin{equation}\label{2.10}
 0<C^{-1}\leq NQ(\Gam_N)^\ell\leq C<\infty
 \end{equation}
 for some constant $C$. Set
 \[
 X^{(N)}_n=\prod_{i=1}^\ell\bbI_{\Gam_N}(\xi_{q_i(n)})\,\,\mbox{and}\,\,
 S_N=\sum_{n=0}^{N-1}X_n^{(N)}.
 \]
 Then
 \begin{equation}\label{2.11}
 d_{TV}(\cL(S_N),\,\mbox{Pois}(\la_N))\to 0\,\,\mbox{as}\,\, N\to\infty
 \end{equation}
 where $\la_N=NQ(\Gam_N)^\ell$. In particular, if
 \begin{equation}\label{2.12}
 \lim_{N\to\infty}NQ(\Gam_N)^\ell=\la
 \end{equation}
 then the distribution of $S_N$ converges in total variation as $N\to\infty$
 to the Poisson distribution with the parameter $\la$.
 \end{corollary}

\subsection{Shifts}\label{subsec2.2}
Our second setup consists of a finite or countable set $\mathcal{A}$,
the sequence space $\Omega=\mathcal{A}^{\mathbb{N}}$,
the $\sigma$-algebra
 $\mathcal{F}$ on $\Omega$ generated by cylinder
sets, the left shift
 $T:\Omega\rightarrow\Omega$, and a $T$-invariant probability measure
 $P$ on $(\Omega,\mathcal{F})$.
We assume that $P$ is $\psi$-mixing
with the $\psi$-dependence coefficient given by (\ref{2.1}) and (\ref{2.2})
considered with respect to the $\sig$-algebras $\cF_{mn},\, n\geq m$ generated
by the cylinder sets $\{\om=(\om_0,\om_1,...)\in\Om:\,\om_i=a_i\,$ for
$m\leq i\leq n\}$ for some $a_m,a_{m+1},...,a_n\in\cA$. Clearly, $\cF_{mn}=
T^{-m}\cF_{0,n-m}$ for $n\geq m$. For each word $a=(a_0,a_1,...,a_{n-1})\in
\cA^n$ we will use the notation $[a]=\{\om=(\om_0,\om_1,...):\, \om_i=a_i,\,
 i=0,1,...,n-1\}$ for the corresponding cylinder set. Write
 $\Omega_{P}$ for the support
of $P$, i.e.
\[
\Om_P=\{\om\in\Om:\,P[\om_0,...,\om_n]>0\,\,\mbox{for all}\,\, n\geq 0\}.
\]
 For $n\ge 1$ set $\mathcal{C}_{n}=\{[w]\::\:w\in\mathcal{A}^{n}\}$.
Since $P$ is $\psi$-mixing it follows (see \cite{KR}, Lemma 3.1)
that there exists $\upsilon>0$ such that
\begin{equation}
P(A)\le e^{-\upsilon n}\text{ for all \ensuremath{n\ge1} and
\ensuremath{A\in\mathcal{C}_{n}}}\:.\label{eq:estimate of measure of cylinders}
\end{equation}
For $n,m\ge1$, $A\in\mathcal{C}_{n}$ and $B\in\mathcal{C}_{m}$ set $n\vee m=
\max\{n,m\}$, $n\wedge m=\min\{n,m\}$,
\[
\pi(A)=\min\{1\le k\le n\::\:A\cap T^{-k}A\ne\emptyset\}
\]
and
\[
\pi(A,B)=\min\{0\le k\le n\wedge m \::\:A\cap T^{-k}B\ne\emptyset
\mbox{ or }B\cap
T^{-k}A\ne\emptyset\}\:.
\]

Let strictly increasing functions $q_{1},...,q_{\ell}:\mathbb{N}\rightarrow
\mathbb{N}$ satisfy (\ref{2.3}) with $q(n)$ defined there.
For each $n\in\mathbb{N}$ define also
\[
\gamma(n)=\min\{k\ge0\::\:q(k)\ge2n\}.
\]

For $\eta\in\Omega$ and $n\ge1$ write $A_{n}^{\eta}=[\eta_{0}...\eta_{n-1}]
\in\mathcal{C}_{n}$.
Let $\tau_{n}^{\eta}:\Omega\rightarrow\mathbb{N}$ be with
\[
\tau_{n}^{\eta}(\omega)=\inf\{k\ge1\::\:T^{q_{i}(k)}\omega\in A_{n}^{\eta}
\text{ for all }1\le i\le\ell\}\:.
\]
For $\eta,\omega\in\Omega$ and $n,m\ge1$ define $\Sig^{\om,\eta}_{n,m}:\Omega
\rightarrow\mathbb{N}$
by
\[
\Sig^{\om,\eta}_{n,m}=\sum_{k=0}^{\tau_{m}^{\eta}-1}\prod_{i=1}^{\ell}
\bbI_{A_{n}^{\omega}}\circ T^{q_{i}(k)}
\]
and write
\[
\kappa^{\omega,\eta}_{n,m}=\min\{\pi(A_{n}^{\omega},A_{m}^{\eta}),
\pi(A_{n}^{\omega}),\pi(A_{m}^{\eta})\}\:.
\]

\begin{theorem}\label{thm:main thm}
There exists a constant $C=C(\ell,\psi(1))\ge1$
such that for every $(\omega,\eta)\in\Omega_{P}\times
\Omega_{P}$
and $n,m\ge1$ with $\psi(m)<(3/2)^{1/(\ell+1)}-1$,
\begin{multline*}
d_{TV}(\mathcal{L}(\Sig^{\om,\eta}_{n,m}),Geo(\frac{P(A_{m}^{\eta})^{
\ell}}{P(A_{m}^{\eta})^{\ell}+P(A_{n}^{\omega})^{\ell}}))\\
\le C\left(e^{-\frac 12\upsilon\kappa^{\omega,\eta}_{n,m}}\gamma(n\vee m)+
\left(1+
\left(\frac{P(A_{n}^{\omega})}{P(A_{m}^{\eta})}\right)^{\ell}
\right)\left((n\vee m)e^{-\frac 12\upsilon\kappa^{\omega,\eta}_{n,m}}
+\psi(m)^{1/2}
\right)
\right)\:.
\end{multline*}
\end{theorem}

\begin{example}
Let us consider an explicit example. Assume $\mathcal{A}=\{0,1,2\}$,
let $\phi:\Omega\rightarrow\mathbb{R}$ be H\" older continuous, and
assume $P$ is the Gibbs measure corresponding to $\phi$.
There exist constants $C>1$ and $\Pi\in\mathbb{R}$ such that for each
$\omega\in\Omega$ and $n\ge1$,
\begin{equation}
C^{-1}\le\frac{P(A_{n}^{\omega})}{\exp(-\Pi n+\sum_{j=0}^{n-1}\phi(T^{j}
\omega))}\le C\:.\label{eq:gibbs property}
\end{equation}
Additionally, it is well known that $P$ is $\psi$-mixing
and that $\psi(m){\rightarrow}0$ as $m\to\infty$ at an exponential speed
(see \cite{Bo}).

Assume $\ell=2$ and that $q_{1}(n)=n$ and $q_{2}(n)=2n$ for each
$n\in\mathbb{N}$. This implies that $\gamma(n)=2n$ for $n\in\mathbb{N}$.
Let $\omega,\eta\in\Omega$ be such that $\omega_{0}=1$, $\eta_{0}=2$,
and $\omega_{j}=\eta_{j}=0$ for each $j\ge1$. It is easy to see
that $\kappa^{\omega,\eta}_{n,n}=n$ for all $n\ge1$. Also, since $T^{j}
\omega=T^{j}\eta$
for $j\ge1$, it follows by (\ref{eq:gibbs property}) that
\[
\sup\{\frac{P(A_{n}^{\omega})}{P(A_{n}^{\eta})}\::\:n\ge1\}<
\infty\:.
\]
Hence, from Theorem \ref{thm:main thm} it follows that
\[
d_{TV}(\mathcal{L}(\Sig^{\om,\eta}_{n,n}),Geo(\frac{P(A_{n}^{\eta})^{
\ell}}{P(A_{n}^{\eta})^{\ell}+P(A_{n}^{\omega})^{\ell}}))
\rightarrow0\text{ exponentially fast as }n\rightarrow\infty\;.
\]
\end{example}
We now return to our general setup. The following corollary deals with
the limit behaviour of $\Sig^{\om,\eta}_{n,m(n)}$ for $P\times P$-typical
pairs $(\omega,\eta)\in\Omega\times\Omega$, where $|m(n)-n|=o(n)$. By $o(n)$ 
we mean an unspecified function $f:\mathbb{N}\rightarrow\mathbb{N}$ with 
$\frac {f(n)}{n}\rightarrow0$ as $n\rightarrow\infty$.
\begin{corollary}\label{cor:cor from main thm}
Let $\{m(n)\}_{n\ge1}\subset\mathbb{N}\setminus\{0\}$ be with $|m(n)-n|=o(n)$
 as $n\rightarrow\infty$. Assume that there exists $\beta\in(0,1)$
and $k\ge1$ such that $\psi(n)=O(\beta^{n})$ and $\gamma(n)=O(n^{k})$
for $n\ge1$. Then for $P\times P$-a.e. $(\omega,\eta)\in
\Omega\times\Omega$,
\[
\underset{n\to\infty}{\lim}\:d_{TV}(\mathcal{L}(\Sig^{\om,\eta}_{n,m(n)}),
Geo(\frac{P(A_{m(n)}^{\eta})^{\ell}}{P(A_{m(n)}^{\eta})^{\ell}+
P(A_{n}^{\omega})^{\ell}}))=0\:.
\]
In particular, if
\[
\lim_{n\to\infty}\frac {P(A_n^\om)}{P(A^\eta_{m(n)})}=\la
\]
then $\cL(\Sig_{n,m(n)}^{\om,\eta})$ converges in total variation as
$n\to\infty$ to the geometric distribution with the parameter 
$(1+\la^\ell)^{-1}$.
\end{corollary}

We observe that, in general (in fact, "usually"), the ratio 
$\frac {P(A_n^\om)}{P(A^\eta_{n})}$ will be unbounded for distinct
$\om,\eta\in\Om$, and so in order to obtain nontrivial limiting
geometric distribution it is necessary to choose cylinders $A_n^\om$
and $A_{m(n)}^\om$ with appropriate lengths. In order to have the ratio
$\frac {P(A_n^\om)}{P(A^\eta_{m(n)})}$
 bounded away from zero and infinity our condition $|m(n)-n|=o(n)$ is,
 essentially, necessary (at least, in the finite entropy case)
 which follows from the Shannon-McMillan-Breiman theorem (see \cite{Pe}).

A number theory (combinatorial) application of our results can be
described in the following way. Let $a,b\in(0,1)$ have base $k$
or continued fraction
expansions with digits $a_0,a_1,...$ and $b_0,b_1,...$,
respectively. For each point $\om\in(0,1)$ with base $k$ or continued fraction
 expansion with digits $\om_0,\om_1,...$ let $\tau^a_m(\om)$ be
the smallest $l\geq 0$ such that the $m$-string $a_0,a_1,...,a_{m-1}$
is repeated by the sequence $\om$ starting from all places $q_i(l),\,
i=1,...,\ell$. Now, count the number $\cN^{b,a}_{n,m}(\om)$ of those
$j<\tau^a_m(\om)$
for which the $n$-string $b_0,b_1,...,b_{n-1}$ is repeated starting from
all places $q_i(j),\, i=1,...,\ell$. Considering on $[0,1)$ the
Lebesgue measure we conclude from our results that for almost all pairs $a,b$
the distribution of $\cN^{b,a}_{n,n}$ in the base $k$ expansion case converges 
in total variation as $n\to\infty$ to the geometric distribution with the 
parameter $1/2$. In the continued fraction case let $G$ be the Gauss 
measure $G(\Gam)=\frac 1{\ln 2}\int_\Gam\frac {dx}{1+x}$ and denote by
$[c_0,c_1,...,c_{n-1}]$ the interval of points $\om\in(0,1)$ having 
continued fraction expansion starting with $c_0,...,c_{n-1}$. Then assuming
that
\[
\lim_{n\to\infty}\frac {G[b_0,...,b_{n-1}]}{G[a_0,...,a_{m(n)-1}]}=\la
\]
and that $\ka^{b,a}_{n,m(n)}\to\infty$ fast enough,
we obtain that the distribution of $\cN^{b,a}_{n,m(n)}$ converges in total
variation to the geometric distribution with the parameter $(1+\la^\ell)^{-1}$.

   \section{Multiple returns for a stationary process}\label{sec3}
\setcounter{equation}{0}

\subsection{A lemma}\label{subsec3.1}
We will need the following result which is, essentially, an exercise in
elementary probability.
\begin{lemma}\label{lem3.1}
Let $Y=\big\{ Y_{k,l}:\, k\geq 0$ and $l\in\{ 0,1\}\big\}$ be independent
Bernoulli random variables such that
$1>P\{ Y_{k,0}=1\}=p=1-P\{ Y_{k,0}=0\}>0$ and $1>P\{ Y_{k,1}=1\}=q=1-
P\{ Y_{k,1}=0\}>0$.
Set $\tau=\min\{ l\geq 0:\, Y_{l,0}=1\}$. Then $S=\sum_{l=0}^{\tau-1}Y_{l,1}$
is a geometric random variable with the parameter $p(p+q-pq)^{-1}$.
\end{lemma}
\begin{proof} Clearly
\[
\{ S=m\}=\bigcup_{n=m}^\infty\{\tau=n\}\cap\{\sum_{l=0}^{n-1}Y_{l,1}=m\}.
\]
Since the processes $\{ Y_{l,1}\}_{l\geq 0}$ and $\{ Y_{l,0}\}_{l\geq 0}$
are independent of each other, the events $\{\tau=n\}$ and
$\{\sum_{l=0}^{n-1}Y_{l,1}=m\}$ are independent, as well. Moreover,
$\tau$ and $\sum_{l=0}^{n-1}Y_{l,1}$ have geometric and binomial distributions,
respectively. Thus,
\begin{eqnarray}\label{3.1}
&P\{ S=m\}=\sum_{n=m}^\infty P\{\tau=n\}P\{\sum_{l=0}^{n-1}Y_{l,1}=m\}\\
&=\sum_{n=m}^\infty(1-p)^np\binom{n}{m}q^m(1-q)^{n-m}\nonumber\\
&=q^m(1-p)^mp\sum_{n=0}^\infty\binom{n+m}{m}(1-p)^n(1-q)^n.\nonumber
\end{eqnarray}
Set $r=(1-p)(1-q)$ then
\begin{equation}\label{3.2}
\sum_{n=0}^\infty\binom{n+m}{n}r^n(1-r)^{m+1}=1
\end{equation}
since we are summing the probability density (mass) function of the negative
binomial distribution with the parameters $(m+1,r)$. Since $\binom {n+m}{n}=
\binom {n+m}{m}$ we obtain from (\ref{3.1}) and (\ref{3.2}) that
\begin{equation}\label{3.3}
P\{ S=m\}=\big(\frac {q(1-p)}{1-r}\big)^m\frac {p}{1-r}=\big(\frac {q(1-p)}
{p+q-pq}\big)^m\frac p{p+q-pq}
\end{equation}
and taking into account that $1-\frac {q(1-p)}{p+q-pq}=\frac p{p+q-pq}$
the proof of the lemma is complete.
\end{proof}

\subsection{Proof of Theorem \ref{thm2.1}}\label{subsec3.2}
Let $X'_{n,\al},\, n=0,1,...,\,\al=0,1$ be a sequence of independent random
variables such that $X'_{n,\al}$ has the same distribution as $X_{n,\al}$.
Set $\tau_M=\min(\tau,M)$,
\[
S'_M=\sum_{n=0}^{M-1}X'_{n,1},\,\tau'=\min\{ n\geq 0:\, X'_{n,0}=1\}\,\,
\mbox{and}\,\,\tau'_M=\min(\tau', M).
\]
Next, Let $Y_{n,0}$ and $Y_{n,1},\, n=0,1,...$ be two independent of each
other sequences of i.i.d. random variables such that
\begin{equation}\label{3.4}
P\{ Y_{n,\al}=1\}=Q(\Gam_\al)^\ell=1-P\{ Y_{n,\al}=0\},\,\al=0,1.
\end{equation}
We can and will assume that all above random variables are defined on the
same (sufficiently large) probability space. Set also
\[
S^*_M=\sum_{n=0}^{M-1}Y_{n,1},\,\tau^*=\min\{ n\geq 0:\, Y_{n,0}=1\}\,\,
\mbox{and}\,\,\tau^*_M=\min(\tau^*,M).
\]

Now observe that $S^*_{\tau^*}$ has by Lemma \ref{lem3.1} the geometric
distribution with the parameter
\begin{equation}\label{3.5}
\varrho=\frac {Q(\Gam_0)^\ell}{Q(\Gam_0)^\ell+Q(\Gam_1)^\ell
(1-Q(\Gam_0)^\ell)}>\rho.
\end{equation}
Next, we can write
\begin{equation}\label{3.6}
d_{TV}(\cL(S_\tau),\,\mbox{Geo}(\rho))\leq A_1+A_2+A_3+A_4+A_5
\end{equation}
where $A_1=d_{TV}(\cL(S_\tau),\,\cL(S_{\tau_M}))$,
 $A_2=d_{TV}(\cL(S_{\tau_M}),\,\cL(S'_{\tau'_M}))$,
 $A_3=d_{TV}(\cL(S'_{\tau'_M}),\,\cL(S^*_{\tau^*_M}))$ ,
 $A_4=d_{TV}(\cL(S^*_{\tau^*_M}),\,\cL(S^*_{\tau^*}))$ and
 $A_5=d_{TV}(\mbox{Geo}(\varrho),\,\mbox{Geo}(\rho))$.

 Introduce random vectors
 $\bfX_{M,\al}=\{ X_{n,\al},\, 0\leq n\leq M\},\,\al=0,1$, $\bfX_M=
 \{\bfX_{M,0},\,\bfX_{M,1}\}$, $\bfX'_{M,\al}=\{ X'_{n,\al},\, 0\leq n\leq M\},
 \,\al=0,1$, $\bfX'_M=\{\bfX'_{M,0},\,\bfX'_{M,1}\}$, $\bfY_{M,\al}=
 \{ Y_{n,\al},\, 0\leq n\leq M\},\,\al=0,1$ and $\bfY_M=\{\bfY_{M,0},\,
 \bfY_{M,1}\}$. Observe that the event $\{ S_\tau\ne S_{\tau_M}\}$ can
 occure only if $\tau>M$. Also, we can write $\{\tau>M\}=\{ X_{n,0}=0\,\,
 \mbox{for all}\,\, n=0,1,...,M\}$ and $\{\tau'>M\}=\{ X'_{n,0}=0\,\,
 \mbox{for all}\,\, n=0,1,...,M\}$ Hence,
 \begin{eqnarray}\label{3.7}
 &A_1\leq P\{\tau>M\}\leq P\{\tau'>M\}+|P\{ X_{n,0}=0\,\,\mbox{for}\,\,
 n=0,1,...,M\}\\
 &-P\{ X'_{n,0}=0\,\,\mbox{for}\,\, n=0,1,...,M\}|\leq P\{\tau'>M\}+
 d_{TV}(\cL(\bfX_{M,0},\,\cL(\bfX'_{M,0}))\nonumber
 \end{eqnarray}
 and similarly,
 \begin{equation}\label{3.8}
 P\{\tau'>M\}\leq P\{\tau^*>M\}+d_{TV}(\cL(\bfX'_{M,0},\,\cL(\bfY_{M,0})).
 \end{equation}
 Since $Y_{n,0},\, n=0,1,...$ are i.i.d. random variables we obtain that
 \begin{equation}\label{3.9}
 P\{\tau^*>M\}=(P\{ Y_{0,0}=0\})^{M+1}=(1-Q(\Gam_0)^\ell)^{M+1}.
 \end{equation}

 Next, we claim that
 \begin{eqnarray}\label{3.10}
 &d_{TV}(\cL(\bfX'_{M,0}),\,\cL(\bfY_{M,0}))\leq d_{TV}(\cL(\bfX'_M),\,
 \cL(\bfY_M))\\
 &\leq\sum_{0\leq n\leq M,\al=0,1}d_{TV}(\cL(X'_{n,\al}),\,\cL(Y_{n,\al})).
 \nonumber \end{eqnarray}
 The first inequality above is clear and the second one holds true in view of
 the following general argument. Let $\mu_1,\mu_2$ and $\tilde\mu_1,
 \tilde\mu_2$ be Borel probability measures on Borel measurable spaces $\cX$
 and $\tilde\cX$, respectively. Then for any product Borel sets
 $U_i\times\tilde U_i\subset\cX\times\tilde\cX,\, i=1,...,k$ such that
 $U_i\subset\cX,\,\tilde U_i\subset\tilde\cX,\, i=1,...,k$ and $U_1,...,U_k$
 are disjoint we have
 \begin{equation*}
 \big\vert\mu_1\times\tilde\mu_1(\cup_{i=1}^k(U_i\times
 \tilde U_i))-\mu_2\times\tilde\mu_2(\cup_{i=1}^k
 (U_i\times\tilde U_i))\big\vert\leq B_1+B_2
 \end{equation*}
 where
 \[
 B_1=\big\vert\sum_{i=1}^k\mu_1(U_i)(\tilde\mu_1(\tilde U_j)-\tilde\mu_2
 (\tilde U_j))\big\vert 
 \]
 and
 \[
 B_2=\big\vert\sum_{i=1}^k\tilde\mu_2(\tilde U_j)(\mu_1( U_i)-\mu_2(U_i))
 \big\vert.
 \]
 Since $U_1,...,U_k$ are disjoint then
 \[
 B_1\leq\sum_{i=1}^k\mu_1(U_i)|\tilde\mu_1(\tilde U_j)-\tilde\mu_2
 (\tilde U_j)|\leq d_{TV}(\tilde\mu_1,\tilde\mu_2) 
 \]
 and
 \[
 B_2\leq\max((\mu_1-\mu_2)(H_+),\,(\mu_2-\mu_1)(H_-))\leq d_{TV}(\mu_1,\mu_2).
 \]
 where $\cX=H_+\cup H_-$ is the Hahn decompositions of $\cX$ into positive
 and negative part with respect to the signed measure $\mu_1-\mu_2$. Thus,
 \[
 |\mu_1\times\tilde\mu_1(W)-\mu_2\times\tilde\mu_2(W)|\leq d_{TV}(\mu_1,
 \mu_2)+d_{TV}(\tilde\mu_1,\tilde\mu_2)
 \]
 for any $W\subset\cX\times\tilde\cX$ having the form $W=\cup_{1\leq i\leq k}
 (U_i\times\tilde U_i)$ with disjoint Borel $U_1,...,U_k\subset\cX$ and 
 arbitrary Borel $\tilde U_1,...,\tilde U_k\subset\tilde\cX$. But any finite
 union of disjoint Borel subsets of $\cX\times\tilde\cX$ can be represented in
  this form, whence the above inequality holds true for all such unions which
  form an algebra of sets. This inequality is preserved under monotone limits
 of sets, and so it remains true for any Borel set $W\subset\cX\times\tilde\cX$
 yielding (\ref{3.10}) by induction on $M$.

 Now,
 \begin{eqnarray}\label{3.11}
 &d_{TV}(\cL(X'_{n,\al},\,\cL(Y_{n,\al}))=|P\{ X'_{n,\al}=1\}-
 P\{ Y_{n,\al}=1\}|\\
 &=|P\{\xi_{q_i(n)}\in\Gam_\al\,\,\mbox{for}\,\, i=1,...,\ell\}-
 Q(\Gam_\al)^\ell|\leq\big((1+\psi(q(n)))^\ell-1\big)Q(\Gam_\al)^\ell
 \nonumber\end{eqnarray}
 where the last inequality follows from Lemma 3.2 in \cite{KR} and it is
 based on standard properties of the $\psi$-mixing coefficient. For any
 positive integers $m<M$ we can write
 \begin{equation}\label{3.12}
 d_{TV}(\cL(\bfX'_M),\,\cL(\bfY_M))\leq(Q(\Gam_0)^\ell+Q(\Gam_1)^\ell))
 \sum_{n=0}^M\big((1+\psi(q(n)))^\ell-1\big).
 \end{equation}
 Observe that
 \begin{equation}\label{3.13}
 d_{TV}(\cL(\bfX_{M,0}),\cL(\bfX'_{M,0}))\leq d_{TV}(\cL(\bfX_{M}),
 \cL(\bfX'_{M}))\,\,\mbox{and}\,\, A_2\leq d_{TV}(\cL(\bfX_{M}),\cL(\bfX'_{M})).
 \end{equation}
 The first inequality in (\ref{3.13}) is clear and the second one follows
 from the fact that $S_{\tau_M}=f(\bfX_M)$ and $S'_{\tau'_M}=f(\bfX'_M)$
 for a certain function $f:\,\{ 0,1\}^{2(M+1)}\to\{ 0,1,...,M\}$. We will
 estimate next $d_{TV}(\cL(\bfX_M),\,\cL(\bfX'_M))$ relying on \cite{AGG}
 warning the reader first that in Section \ref{sec2} we defined $d_{TV}$ in
  a more standard way than in \cite{AGG} where this quantity is multiplied
  by the factor 2, and so we adjust estimates from there accordingly.

  By Theorem 3 in \cite{AGG},
  \begin{equation}\label{3.14}
  d_{TV}(\cL(\bfX_M),\,\cL(\bfX'_M))\leq 2b_1+2b_2+2b_3+2\sum_{0\leq n\leq M,
  \al=0,1}p^2_{n,\al}
  \end{equation}
  where for $\al=0,1$,
  \begin{equation}\label{3.15}
  p_{n,\al}=P\{ X_{n,\al}=1\}=P\{\xi_{q_i(n)}\in\Gam_\al\,\,\mbox{for}\,\,
  i=1,...,\ell\}\leq (1+\psi(q(n)))^\ell Q(\Gam_\al)^\ell
  \end{equation}
  with the latter inequality satisfied by Lemma 3.2 in \cite{KR}. In order to
  define $b_1,b_2$ and $b_3$ we introduce the distance between positive 
  integers
  \[
  \del(k,l)=\min_{1\leq i,j\leq\ell}|q_i(k)-q_j(l)|
  \]
  and the set
  \[
  B^{M,R}_{n,\al}=\{(l,\al),\,(l,1-\al):\, 0\leq l\leq M,\,\del(l,n)\leq R\}
  \]
  (which, in fact, does not depend on $\al$)
  where an integer $R>0$ is another parameter. Set also
  $I_M=\{(n,\al):\, 0\leq n\leq M,\,\al=0,1\}$. Then
  \begin{equation}\label{3.16}
  b_1=\sum_{(n,\al)\in I_M}\sum_{(l,\be)\in B_{n,\al}^{M,R}}p_{n,\al}
  p_{l,\be},
  \end{equation}
  \begin{equation}\label{3.17}
  b_2=\sum_{(n,\al)\in I_M}\sum_{(n,\al)\ne(l,\be)\in B_{n,\al}^{M,R}}
  p_{(n,\al),(l,\be)},
  \end{equation}
 where $p_{(n,\al),(l,\be)}=E(X_{n,\al}X_{l,\be})$, and
 \begin{equation}\label{3.18}
 b_3=\sum_{(n,\al)\in I_M}s_{n,\al}
 \end{equation}
 where
 \[
 s_{n,\al}=E\big\vert E\big(X_{n,\al}-p_{n,\al}|\sig\{ X_{l,\be}:\,(l,\be)\in
 I_M\setminus B^{M,R}_{n,\al}\}\big)\big\vert.
 \]
  Since the functions $q_i,\, i=1,...,\ell$ are strictly increasing, for
  any $i,j,n$ and $k$ there exists at most one $l$ such that $q_i(n)-q_j(l)
  =k$. It follows from here that
  \begin{equation}\label{3.19}
  |B^{M,R}_{n,\al}|\leq 8\ell^2(R+1)
  \end{equation}
  where $|U|$ denotes the cardinality of a finite set $U$. It follows from
  (\ref{3.15}), (\ref{3.16}) and (\ref{3.19}) that
  \begin{equation}\label{3.20}
  b_1\leq 4(M+1)\ell^2(R+1)(1+\psi(1))^{2\ell}(Q(\Gam_0)^{2\ell}+
  Q(\Gam_1)^{2\ell}).
  \end{equation}

  Next,
  \begin{equation}\label{3.21}
  p_{(n,\al),(l,\be)}=P\{ X_{n,\al}=X_{l,\be}=1\}=0
  \end{equation}
  if $n=l$ and $\be=1-\al$ since $\Gam_0\cap\Gam_1=\emptyset$. If $n\ne l$ then
   assuming, for instance, that $l>n$ we obtain by Lemma 3.2 in \cite{KR} that
  \begin{eqnarray}\label{3.22}
  &p_{(n,\al),(l,\be)}=P\{ X_{n,\al}=X_{l,\be}=1\}\\
  &\leq P\{ X_{n,\al}=1\,\,\mbox{and}\,\,\xi_{q_\ell(l)}\in\Gam_\be\}
  \leq(1+\psi(1))^{\ell+1}Q(\Gam_\al)^\ell Q(\Gam_\be).\nonumber
  \end{eqnarray}
  Hence,
  \begin{equation}\label{3.23}
  b_2\leq 2(M+1)\ell^2(R+1)(1+\psi(1))^{\ell+1}( Q(\Gam_0)^\ell
  +Q(\Gam_1)^\ell)(Q(\Gam_0)+Q(\Gam_1)).
  \end{equation}

  Next, we claim that
  \begin{eqnarray}\label{3.24}
  &s_{n,\al}\leq 2^{2(\ell+2)}(2-(1+\psi(R))^{\ell+1})^{-2}\psi(R)E|X_{n,\al}
  -p_{n,\al}|\\
  &\leq 2^{2\ell+5}(2-(1+\psi(R))^{\ell+1})^{-2}\psi(R)p_{n,\al}
  \nonumber\end{eqnarray}
  where $s_{n,\al}$ is the same as in (\ref{3.18}). Indeed, let $\cG$ be the
  $\sig$-algebra generated by all $\xi_{q_i(l)},\, i=1,...,\ell$ such that
  $(l,0)\in I_M\setminus B^{M,R}_{n,\al}$ and $\cH$ be the $\sig$-algebra
  generated by $\xi_{q_i(n)},\, i=1,...,\ell$. Since $|q_i(n)-q_j(l)|>R$
  for all $i,j=1,...,\ell$ and $l$ such that $(l,0)\in I_M\setminus
  B^{M,R}_{n,\al}$ we conclude from Lemma 3.3 in \cite{KR} that
  \begin{equation}\label{3.25}
  \psi(\cG,\cH)\leq 2^{2(\ell+2)}\psi(R)(2-(1+\psi(R))^{\ell+1})^{-2}
  \end{equation}
  provided $\psi(R)<2^{\frac 1{\ell+1}}-1$ which we assume. Since
  $\sig\{ X_{l,\be}:\,(l,\be)\in I_M\setminus B^{M,R}_{n,\al}\}\subset\cG$
  and $\sig\{ X_{n,\al}\}\subset\cH$ we obtain (\ref{3.24}) from (\ref{2.1})
  and (\ref{3.25}). Now by (\ref{3.15}), (\ref{3.18}) and (\ref{3.24}),
  \begin{equation}\label{3.26}
  b_3\leq 2^{2\ell+5}(M+1)(2-(1+\psi(R))^{\ell+1})^{-2}(1+\psi(1))^\ell\psi(R)
  (Q(\Gam_0)^\ell+Q(\Gam_1)^\ell)).
  \end{equation}

  Next, in the same way as in the estimate of $A_2$ we conclude that
  \begin{equation}\label{3.27}
  A_3\leq d_{TV}(\cL(\bfX'_M),\cL(\bfY_M))
  \end{equation}
  which together with (\ref{3.12}) estimates $A_3$.

  As in the estimate of $A_1$ we see that
  \begin{equation}\label{3.28}
  A_4\leq P\{\tau^*>M\}\leq (1-Q(\Gam_0)^\ell)^{M+1}
  \end{equation}
  since $Y_{n,0},\, n=0,1,...$ are i.i.d. random variables.

  Since $\varrho>\rho$ we obtain
  \begin{eqnarray}\label{3.29}
  &A_5\leq\sum_{k=0}^\infty |\varrho(1-\varrho)^k-\rho(1-\rho)^k|\leq
  2\sum_{k=1}^\infty((1-\rho)^k-(1-\varrho)^k)\\
  &=2(1-\rho)\rho^{-1}-2(1-\varrho)\varrho^{-1}=
  2\frac {\varrho-\rho}{\rho\varrho}=2Q(\Gam_1)^\ell. \nonumber
  \end{eqnarray}
 Collecting (\ref{3.6})--(\ref{3.15}), (\ref{3.20}). (\ref{3.23}),
 (\ref{3.24}) and (\ref{3.26})--(\ref{3.29}) we derive (\ref{2.5}).
  \qed

 In order to prove Corollary \ref{cor2.2} we rely on the estimate (\ref{2.5})
 with $\Gam_0=\Gam_N$ and $\Gam_1=\Del_N$ choosing $M=M_N\to\infty$ and
 $R=R_N\to\infty$ as $N\to\infty$ so that
 \begin{eqnarray}\label{3.30}
 &\lim_{N\to\infty}M_NQ(\Gam_N)^\ell=\infty,\,\lim_{N\to\infty}Q(\Gam_N)^\ell
 \sum_{n=0}^{M_N}\psi(q(n))=0,\\
 &\lim_{N\to\infty}M_N\psi(R_N)Q(\Gam_N)^\ell=0\,\,\mbox{and}\,\,
 \lim_{N\to\infty}M_NR_NQ(\Gam_N)^{\ell+1}=0\nonumber
 \end{eqnarray}
 which is clearly possible since $\psi(n)\to 0$ and $q(n)\to\infty$ as
 $n\to\infty$. This together with (\ref{2.5}) yields (\ref{2.7}).
 \qed

 \subsection{Returns until a fixed time}\label{subsec3.3}
 Now we will prove Theorem \ref{thm2.3}. By Theorem 1 in \cite{AGG},
 \begin{equation}\label{3.31}
 d_{TV}(\cL(S_N),\,\mbox{Pois}(ES_N))\leq b_1+b_2+b_3
 \end{equation}
 where $b_1,b_2$ and $b_3$ are defined by (\ref{3.16})--(\ref{3.18}) with
 the sums there taken only in $n$ and $l$ (but not in $\al$), taking $N$ in
 place of $M$ and replacing
 there $p_{n,\al}$ by $p_n=P\{ X_n=1\}$, $p_{(n,\al),(l,\be)}$ by
 $p_{nl}=E(X_nX_l)$, $B_{n,\al}^{M,R}$ by $B_n^{N,R}=\{ l:\, 0\leq l\leq N,\,
 \del(l,n)\leq R\}$ and $s_{n,\al}$ by $s_n=E|E(X_n-p_n|\sig\{ X_l:\,
 l\in I_N\setminus B^{N,R}_n)|$ where $I_N=\{ 0,1,...,N\}$. Then all right
 hand side estimates (\ref{3.20}), (\ref{3.23}) and (\ref{3.26}) remain valid
 but we will have to consider only one set $\Gam_0=\Gam$ (deleting terms
 with $Q(\Gam_1)$ there) and in order to complete the proof of Theorem
 \ref{2.3} it remains to show that
 \begin{equation}\label{3.32}
 d_{TV}\big(\mbox{Pois}(ES_N),\,\mbox{Pois}(NQ(\Gam)^\ell)\big)\leq
 \wp\big(CQ(\Gam)^\ell\sum_{n=0}^N\psi(q(n))\big).
 \end{equation}

 Indeed, for any $\la_1,\la_2>0$ by Lemma 3.4 in \cite{KR},
 \begin{eqnarray}\label{3.33}
 &d_{TV}(\mbox{Pois}(\la_1),\,\mbox{Pois}(\la_2))\\
 &\leq\sum_{n=0}^\infty |e^{-\la_1}\frac {\la_1^n}{n!}-e^{-\la_2}
 \frac {\la_2^n}{n!}|\leq 2|\la_1-\la_2|e^{-|\la_1-\la_2|}.\nonumber
 \end{eqnarray}
 Now,
 \[
 ES_N=\sum_{n=0}^{N-1}EX_n\,\,\mbox{and}\,\, EX_n=P\{ X_n=1\}=P\{\xi_{q_i(n)}
 \in\Gam\,\,\mbox{for}\,\, i=1,...,\ell\}.
 \]
 By Lemma 3.2 in \cite{KR} (which is an easy application of the definitions
 (\ref{2.1}) and (\ref{2.2}) of the $\psi$-dependence coefficient) together
 with stationarity of the sequence $\{\xi_n\}$ we obtain
 \begin{equation}\label{3.34}
 |P\big\{\cap_{i=1}^\ell\{\xi_{q_i(n)}\in\Gam\}\big\}-Q(\Gam)^\ell|\leq
 \big((1+\psi(q(n)))^\ell-1\big)Q(\Gam)^\ell.
 \end{equation}
 Hence,
 \begin{equation}\label{3.35}
 |ES_N-NQ(\Gam)^\ell|\leq Q(\Gam)^\ell\sum_{n=0}^N\big((1+\psi(q(n)))^\ell
 -1\big)\leq CQ(\Gam)^\ell\sum_{n=0}^N\psi(q(n)),
 \end{equation}
 where $C>0$ does not depend on $N$ or $\Gam$, and (\ref{3.32}) follows.
 \qed

 In order to prove Corollary \ref{cor2.4} we rely on (\ref{2.9}) with
 $\Gam=\Gam_N$ and choosing $R=R_N\to\infty$ as $N\to\infty$ so that
 $\lim_{N\to\infty}R_NQ(\Gam_N)=0$. In view of (\ref{2.10}) and
 taking into account that $\psi(n)\to 0$ and $q(n)\to\infty$ as $n\to\infty$
 we obtain that
 \[
 \lim_{N\to\infty}\wp\big(CQ(\Gam_N)^\ell\sum_{n=0}^N\psi(q(n))\big)=0
 \]
 which together with (\ref{2.9}) yields (\ref{2.11}).
    \qed

 \section{Returns to cylinder sets for shifts}\label{sec4}\setcounter{equation}
{0}

\subsection{Preliminary lemmas and Corollary \ref{cor:cor from main thm}}
\label{subsec4.1}
First we prove Corollary \ref{cor:cor from main thm} while relaying
on Theorem \ref{thm:main thm}, for which we need the following lemma. In what 
follows $\{m(n)\}_{n\ge1}$ is a sequence of positive integers with $|m(n)-n|
=o(n)$ as $n\rightarrow\infty$. For $n\ge1$ we write $b(n)=n\wedge m(n)$.
\begin{lemma}
\label{lem:bound on periods}Set $c=3\upsilon^{-1}$ and let $\mathcal{E}$
be the set of all $(\omega,\eta)\in\Omega\times\Omega$ for which
there exists $N=N(\omega,\eta)\ge1$ such that $\kappa^{\omega,\eta}_{n,m(n)}
\ge
b(n)-c\ln b(n)$
for all $n\ge N$, then $P\times P(\Omega^{2}\setminus
\mathcal{E})=0$.
\end{lemma}

\begin{proof}
For $\omega\in\Omega$ and $n\ge1$ set
\[
B_{\omega,n}=\{\eta\in\Omega\::\:\pi(A_{n}^{\omega},A_{m(n)}^{\eta})\le b(n)-
c\ln b(n)\}.
\]
Assume $b(n)-c\ln b(n)\ge1$ and set $d=[b(n)-c\ln b(n)]$, then
\[
P(B_{\omega,n})\le\sum_{r=0}^{d}P\{\eta\::\:
T^{-r}A_{n}^{\omega}\cap A_{m(n)}^{\eta}\ne\emptyset\}+\sum_{r=0}^{d}P
\{\eta\::\:T^{-r}A_{m(n)}^{\eta}\cap A_{n}^{\omega}\ne\emptyset\}.
\]
For $0\le r\le d$,
\[
\{\eta\::\:T^{-r}A_{n}^{\omega}\cap A_{m(n)}^{\eta}\ne\emptyset\}
=T^{-r}[\omega_{0},...,\omega_{n\wedge (m(n)-r)-1}]
\]
and
\[
\{\eta\::\:T^{-r}A_{m(n)}^{\eta}\cap A_{n}^{\omega}\ne\emptyset\}
=[\omega_{r},...,\omega_{n\wedge(m(n)+r)-1}].
\]
Hence by (\ref{eq:estimate of measure of cylinders}),
\begin{eqnarray*}
P(B_{\omega,n}) & \le &
\sum_{r=0}^{d}e^{-\upsilon(n\wedge (m(n)-r))}
+\sum_{r=0}^{d}e^{-\upsilon((n-r)\wedge m(n))}\\& \le &
2\sum_{r=0}^{d}e^{-\upsilon(b(n)-r)}\le
2\frac{e^{-\upsilon(b(n)-d)}}{1-e^{-\upsilon}}\le
\frac{2b(n)^{-3}}{1-e^{-\upsilon}}.
\end{eqnarray*}
From this and since $|b(n)-n|=o(n)$ it follows that $\sum_{n=1}^{\infty}
P(B_{\omega,n})<\infty$,
and so by the Borel-Cantelli lemma
\[
P\{\eta\::\:\#\{n\ge1\::\:\eta\in B_{\omega,n}\}=\infty\}=0\:.
\]
From Fubini's theorem we now get,
\begin{multline*}
P\times P\{(\omega,\eta)\::\:\#\{n\ge1\::\:\pi(A_{n}^{\omega},
A_{m(n)}^{\eta})\le b(n)-c\ln b(n)\}=\infty\}\\
=\int_{\Omega}P\{\eta\::\:\#\{n\ge1\::\:\eta\in B_{\omega,n}\}=
\infty\}\:dP(\omega)=0\:.
\end{multline*}
In a similar manner (see \cite[Corollary 2.2]{KR}) it can be shown
that
\[
P\{\omega\::\:\#\{n\ge1\::\:\pi(A_{n}^{\omega})\le b(n)-c\ln b(n)\}=\infty\}
=0\:
\]
and
\[
P\{\eta\::\:\#\{n\ge1\::\:\pi(A_{m(n)}^{\eta})\le b(n)-c\ln b(n)\}=\infty\}
=0\:
.\]
This completes the proof of the lemma.
\end{proof}
\begin{proof}[Proof of Corollary \ref{cor:cor from main thm}]
Let $c$ and $\mathcal{E}$ be as in the statement of Lemma \ref{lem:bound on
 periods}.
Denote by $h$ the entropy of the system $(\Omega,T,P)$.
Let $\mathcal{E}_{0}$ be the set of all $(\omega,\eta)\in\mathcal{E}
\cap(\Omega_{P}\times\Omega_{P})$
for which
\[
-\underset{n\to\infty}{\lim}\frac{\log P(A_{n}^{\omega})}{n}=
-\underset{n\to\infty}{\lim}
\frac{\log P(A_{n}^{\eta})}{n}=h\:.
\]
By the Shannon-McMillan-Breiman Theorem (see, for instance, \cite{Pe})
and Lemma \ref{lem:bound on periods}
it follows that $P\times P(\Omega^{2}\setminus\mathcal{E}_{0})
=0$.

Let $(\omega,\eta)\in\mathcal{E}_{0}$, then for every $n\ge1$ large
enough
\[
e^{-\upsilon\kappa^{\omega,\eta}_{n,m(n)}/2}\le\exp(\frac{-\upsilon(b(n)-c\ln
 b(n))}{2})\le
b(n)^{2}\cdot e^{-\upsilon b(n)/2}\:.
\]
By our assumption $\gamma(n)$ grows at most polynomially, hence by $|m(n)-n|=
o(n)$,
\[
e^{-\upsilon\kappa^{\omega,\eta}_{n,m(n)}/2}\gamma(n\vee m(n))\overset{n}
{\rightarrow}0\text{
as }n\rightarrow\infty\:.
\]
From
\begin{multline*}
\underset{n\to\infty}{\lim}\:\log\left(\left(\frac{P(A_{n}^{\omega})}{
P(A_{m(n)}^{\eta})}\right)^{\ell}(n\vee m(n))b(n)^{2}e^{-\upsilon b(n)/2}
\right)\\
=\underset{n\to\infty}{\lim}\:n\cdot\left(\frac{\ell\log P(A_{n}^{\omega
})}{n}-\frac{\ell\log P(A_{m(n)}^{\eta})}{n}+\frac{\log (n\vee m(n))}{n}+
\frac{2\log b(n)}{n}-\frac{\upsilon b(n)}{2n}\right)\\
=\left(\ell h-\ell h-\frac{\upsilon}{2}\right)\underset{n\to\infty}{\lim}\:n=
-\infty,
\end{multline*}
it follows that
\[
\left(1+\left(\frac{P(A_{n}^{\omega})}{P(A_{m(n)}^{\eta})}
\right)^{\ell}\right)(n\vee m(n))e^{-\upsilon\kappa^{\omega,\eta}_{n,m(n)}/2}
\rightarrow 0
\text{ as }n\rightarrow\infty\:.
\]
By our assumption $\psi(n)$ tends to $0$ at an exponential rate as
$n\rightarrow\infty$, hence we also have
\[
\left(1+\left(\frac{P(A_{n}^{\omega})}{P(A_{m(n)}^{\eta})}
\right)^{\ell}\right)\psi(m(n))^{1/2}\rightarrow 0\text{ as }n\rightarrow\infty
\:.
\]
The corollary now follows directly from Theorem \ref{thm:main thm}.
\end{proof}
 In what follows
we will consider $\ell$ and $\psi(1)$ as global constants. Hence,
whenever we use the big-O notation the implicit constant may depend
on these parameters. We will need also the following result.
\begin{lemma}
\label{lem:exp estimate}Let $t\ge1$ and $n\ge1$ be such that $\psi(n)
<(3/2)^{1/(\ell+1)}-1$.
Then for every $\eta\in\Omega_{P}$,
\[
\left|P\left\{ P(A_{n}^{\eta})^{\ell}\tau_{n}^{\eta}>t
\right\} -\mbox{\normalfont Pois}(t)\{0\}\right|=O\left(te^{-\upsilon
\pi(A_{n}^{\eta})}\left(n+\gamma(n)\right)+t\psi(n)\right),
\]
where, recall, {\normalfont Pois}$(t)$ is the Poisson distribution with the
parameter $t$.
\end{lemma}

\begin{proof}
For $n\ge1$ set $N_{n}=[tP(A_{n}^{\eta})^{-\ell}]$ and
\[
S_{n}=\sum_{k=1}^{N_{n}}\prod_{i=1}^{\ell}\bbI_{A_{n}^{\eta}}\circ T^{q_{i}(k)},
\]
then
\begin{equation}
P\{S_{n}=0\}=P\left\{ P(A_{n}^{\eta})^{\ell}
\tau_{n}^{\eta}>t\right\} \:.\label{eq:S and tau relation}
\end{equation}
By \cite[Theorem 2.1]{KR},
\[
\left|P\{S_{n}=0\}-\mbox{Pois}(t)\{0\}\right|=O\left(te^{-\upsilon
\pi(A_{n}^{\eta})}\left(n+\gamma(n)\right)+t\psi(n)\right)\:.
\]
This together with (\ref{eq:S and tau relation}) proves the lemma.
\end{proof}
\subsection{Proof of Theorem \ref{thm:main thm}}\label{subsec4.2}
Let $(\omega,\eta)\in\Omega_{P}\times\Omega_{P}$
and $n,m\ge1$ with $\psi(m)<(3/2)^{1/(\ell+1)}-1$. Set $\kappa=\kappa^{\omega,
\eta}_{n,m}$, then we can clearly assume that
\begin{equation}
e^{-\upsilon\kappa/2}\left(n\vee m+\gamma(n\vee m)\right)\le\frac{1}{2}\:.
\label
{eq:trivial bound}
\end{equation}
Set $\epsilon=\max\{e^{-\upsilon\kappa},\psi(m)\}$ and $t=\epsilon^{-1/2}$,
then $0<\epsilon\leq 1$ and $e^{-t}<\epsilon$.

Set $p_{\eta}=P(A_{m}^{\eta})^{\ell}$, $p_{\omega}=P
(A_{n}^{\omega})^{\ell}$,
$L=[t\cdot p_{\eta}^{-1}]$, $I_{0}=\{\gamma(n\vee m),...,L\}\times\{0\}$,
$I_{1}=\{\gamma(n\vee m),...,L\}\times\{1\}$, and $I=I_{0}\cup I_{1}$.
For $\gamma(n\vee m)\le k\le L$ define random variables $X_{k,0}$ and
$X_{k,1}$ on $(\Omega,\mathcal{F},P)$ by
\[
X_{k,0}=\prod_{i=1}^{\ell}\bbI_{A_{m}^{\eta}}\circ T^{q_{i}(k)}\text{ and }
X_{k,1}=\prod_{i=1}^{\ell}\bbI_{A_{n}^{\omega}}\circ T^{q_{i}(k)},
\]
and denote the Bernoulli process $\{X_{k,l}\::\:(k,l)\in I\}$, i.e.
$X_{k,l}$ takes values 0 or 1 only, by
$\mathbf{X}$. Let
\[
\mathbf{X}'=\{X'_{k,l}\::\:(k,l)\in I\}
\]
be a Bernoulli process, with $\mathcal{L}(X_{k,l})=\mathcal{L}(X'_{k,l})$
for each $(k,l)\in I$, such that the $X'_{k,l}$ are all mutually
independent.

Let $\mathbf{Y}=\{Y_{k,l}\::\: k\geq 0,\, l=0,1\}$ be a collection of
independent
Bernoulli random variables such that $\bbP\{ Y_{k,0}=1\}=p_\eta=1-\bbP\{
Y_{k,0}=0\}$ and $\bbP\{ Y_{k,1}=1\}=p_\om=1-\bbP\{Y_{k,1}=0\}$. Write also,
\[
\mathbf{Y}'=\{Y_{k,l}\::\:(k,l)\in I\}\:.
\]

For $y\in\{0,1\}^{\mathbb{N}\times\{0,1\}}$ set
\[
\tilde{f}(y)=\inf\{k\ge0\::\:y_{k,0}=1\}\mbox{ and }\tilde{g}(y)=
\sum_{j=0}^{\tilde{f}(y)-1}y_{j,1},
\]
and for $y\in\{0,1\}^{I}$ set
\[
f(y)=\min\{\gamma(n\vee m)\le k\le L\::\:y_{k,0}=1\mbox{ or }k=L\}\mbox{ and }
g(y)=
\sum_{j=\gamma(n\vee m)}^{f(y)-1}y_{j,1}\:.
\]

Let $S\subset\mathbb{N}$, then
\begin{multline}
|P\{\Sig^{\om,\eta}_{n,m}\in S\}-\mbox{Geo}(\frac{p_{\eta}}{p_{\eta}+
p_{\omega}})(S)|\le|P\{\Sig^{\om,\eta}_{n,m}\in S\}-P
\{g(\mathbf{X})\in S\}|\\
+|P\{g(\mathbf{X})\in S\}-P\{g(\mathbf{X}')\in S\}|+
|P\{g(\mathbf{X}')\in S\}-P\{g(\mathbf{Y}')\in S\}|\\
+|P\{g(\mathbf{Y}')\in S\}-P\{\tilde{g}(\mathbf{Y})\in S\}|
+|P\{\tilde{g}(\mathbf{Y})\in S\}-\mbox{Geo}(\frac{p_{\eta}}{p_{\eta}+
p_{\omega}})(S)|\\
=\sigma_{1}+\sigma_{2}+\sigma_{3}+\sigma_{4}+\sigma_{5}\:.\label{eq:5
expressions}
\end{multline}

Let us estimate $\sigma_{1}$ from above. The event $\{\Sig^{\om,\eta}_{n,m}\ne
g(\mathbf{X})\}$
is contained in the union of the events $\{\tau_{m}^{\eta}>L\}$ and
\[
E=\{\prod_{i=1}^{\ell}\bbI_{A_{n}^{\omega}}\circ T^{q_{i}(k)}=1\text{ or }
\prod_{i=1}^{\ell}\bbI_{A_{m}^{\eta}}\circ T^{q_{i}(k)}=1\text{ for some }0\le
k<\gamma(n\vee m)\}\:.
\]
By (\ref{eq:estimate of measure of cylinders}),
\[
P(E)=O(\gamma(n\vee m)e^{-\upsilon (n\wedge m)})\:.
\]

Since Pois$(t)\{ 0\}=e^{-t}<\ve$, this together with Lemma
\ref{lem:exp estimate} yields
\[
\sigma_{1}=O\left(\epsilon+te^{-\upsilon\kappa}\left(m+\gamma(n\vee m)\right)+
t\psi(m)\right)\:.
\]

Note that $\sigma_{2}\le d_{TV}(\mathcal{L}(\mathbf{X}),\mathcal{L}
(\mathbf{X}'))$,
hence the following lemma gives an upper estimate on $\sigma_{2}$.
\begin{lemma}
\label{lem:estimate on d(X,X')}It holds that,
\[
d_{TV}(\mathcal{L}(\mathbf{X}),\mathcal{L}(\mathbf{X}'))=O\left(t\left(1+
\frac{p_{\omega}}{p_{\eta}}\right)\left((n\vee m)e^{-\upsilon\kappa}+\psi(n
\vee m)\right)\right)\:.
\]
\end{lemma}

\begin{proof}[Proof of Lemma \ref{lem:estimate on d(X,X')}]
For $(k,l)\in I$ set
\[
B_{k,l}=\cup_{i,j=1}^{\ell}\{(r,s)\in I\::\:|q_{i}(r)-q_{j}(k)|\le2(n\vee m)\}
\]
and
\[
\mathcal{G}_{k,l}=\sigma\left\{ X_{r,s}\::\:(r,s)\in I\setminus B_{k,l}\right\}
 \:.
\]
By Theorem 3 in \cite{AGG},
\[
d_{TV}(\mathcal{L}(\mathbf{X}),\mathcal{L}(\mathbf{X}'))\le O(b_{1}+b_{2}+
b_{3}),
\]
where
\[
b_{1}=\sum_{(k,l)\in I}\:\sum_{(r,s)\in B_{k,l}}P\{X_{k,l}=1\}
P\{X_{r,s}=1\},
\]
\[
b_{2}=\sum_{(k,l)\in I}\:\sum_{(r,s)\in B_{k,l}\setminus(k,l)}P
\{X_{k,l}=1=X_{r,s}\},
\]
and
\[
b_{3}=\sum_{(k,l)\in I}E\left|E\left[X_{k,l}-E[X_{k,l}]\mid\mathcal{G}_{k,l}
\right]\right|\:.
\]

Let us estimate $b_{1}$ from above. By (\ref{eq:estimate of measure of
cylinders}),
\[
P\{X_{k,l}=1\}\le e^{-\upsilon (n\wedge m)}\text{ for }(k,l)\in I\:.
\]
Also, by the $\psi$-mixing assumption (see \cite[Lemma 3.2]{KR}),
\[
P\{X_{k,l}=1\}=\begin{cases}
O(p_{\eta}) & \text{if }l=0\\
O(p_{\omega}) & \text{if }l=1
\end{cases}\text{ for }(k,l)\in I\:.
\]
Since $|B_{k,l}|=O(n\vee m)$ for $(k,l)\in I$,
\[
b_{1}=O\left(L(n\vee m)\cdot e^{-\upsilon (n\wedge m)}(p_{\eta}+p_{\omega})
\right)\:.
\]
Hence by $Lp_{\eta}\le t$,
\[
b_{1}=O\left(t(n\vee m)\cdot e^{-\upsilon (n\wedge m)}(1+\frac{p_{\omega}}
{p_{\eta}})\right)\:.
\]

We shall now estimate $b_{2}$. Let $(k,l)\in I$ and $(r,s)\in B_{k,l}\setminus
(k,l)$
be given. Assume without loss of generality that $k\ge r$. If $|q_{1}(r)-q_{1}
(k)|<\kappa$
then $\{X_{k,l}=1=X_{r,s}\}=\emptyset$ by the definition of $\kappa$.
Otherwise, by the $\psi$-mixing assumption and (\ref{eq:estimate of measure of
 cylinders}),
\[
P\{X_{k,l}=1=X_{r,s}\}=O\left(P\{X_{k,l}=1\}\cdot e^{-\upsilon
\kappa}\right)\:.
\]
Hence, by the considerations made for bounding $b_{1}$,
\[
b_{2}=O\left(t(n\vee m)\cdot e^{-\upsilon\kappa}(1+\frac{p_{\omega}}{p_{\eta}})
\right)
\:.
\]

Finally, we estimate $b_{3}$ from above. Given $(k,l)\in I$ it follows,
by the argument given in the proof of \cite[Theorem 2.1]{KR} in
order to estimate $b_{3}$, that
\[
E\left|E\left(X_{k,l}-EX_{k,l}\mid\mathcal{G}_{k,l}\right)\right|=O\left(
\psi(n\vee m)P\{X_{k,l}=1\}\right)\:.
\]
Hence,
\[
b_{3}=O\left(t\psi(n\vee m)(1+\frac{p_{\omega}}{p_{\eta}})\right)\:.
\]
Now by the estimates on $b_{1}$, $b_{2}$ and $b_{3}$ the lemma
follows.
\end{proof}
We now resume the main proof and estimate $\sigma_{3}$. As explained in
Section \ref{sec3}, given
probability distributions $\mu_{1},\mu_{2},\nu_{1},\nu_{2}$, on the
same measurable space, it holds that
\[
d_{TV}(\mu_{1}\times\nu_{1},\mu_{2}\times\nu_{2})\le d_{TV}(\mu_{1},\mu_{2})
+d_{TV}(\nu_{1},\nu_{2})\:.
\]
From this and since $\mathbf{X}'$ and $\mathbf{Y}'$ are independent
Bernoulli precesses,
\begin{multline*}
\sigma_{3}\le d_{TV}(\mathcal{L}(\mathbf{X}'),\mathcal{L}(\mathbf{Y}'))\le
\sum_{(k,l)\in I}d_{TV}(\mathcal{L}(X'_{k,l}),\mathcal{L}(Y_{k,l}))\\
=\sum_{(k,l)\in I}|P\{X'_{k,l}=1\}-P\{Y_{k,l}=1\}|\:.
\end{multline*}
For $\gamma(n\vee m)\le k\le L$ it follows by the $\psi$-mixing assumption
(see \cite[Lemma 3.2]{KR}) that
\[
|P\{X'_{k,0}=1\}-P\{Y_{k,0}=1\}|=\left|P\left\{
 \prod_{i=1}^{\ell}\bbI_{A_{m}^{\eta}}\circ T^{q_{i}(k)}=1\right\} -p_{\eta}
 \right|=O(\psi(n\vee m)p_{\eta}),
\]
and similarly
\[
|P\{X'_{k,1}=1\}-P\{Y_{k,1}=1\}|=O(\psi(n\vee m)p_{\omega})\:.
\]
Hence,
\[
\sigma_{3}=O(L\psi(n\vee m)(p_{\eta}+p_{\omega}))=O(t\psi(n\vee m)(1+
\frac{p_{\omega}}
{p_{\eta}}))\:.
\]

In order to bound $\sigma_{4}$, note that the event $\{g(\mathbf{Y}')\ne
\tilde{g}(\mathbf{Y})\}$
is contained in the union of the events
\[
E_{1}=\{Y_{k,l}=1\text{ for some }0\le k<\gamma(n\vee m)\text{ and }l=0\text{ 
or }
1\}
\]
and
\[
E_{2}=\{Y_{k,0}=0\mbox{ for each }\gamma(n\vee m)\le k\le L\}.
\]
By (\ref{eq:estimate of measure of cylinders}),
\[
P(E_{1})=O(\gamma(n\vee m)e^{-\upsilon (n\wedge m)})\:.
\]
Since $\mathbf{Y}$ is an independent Bernoulli process with $P\{ Y_{k,0}
=1\}=p_{\eta}$ for $k\ge0$,
\[
P(E_{2})=O\left((1-p_{\eta})^{L-\gamma(n\vee m)}\right)\:.
\]
Since $x\ge\log(1+x)$ for $x>-1$,
\begin{eqnarray*}
&(1-p_{\eta})^{L-\gamma(n\vee m)}=\exp\left(\left(\gamma(n\vee m)-L\right)\log
\left(1+\frac{p_{\eta}}{1-p_{\eta}}\right)\right)\\
&\le\exp\left(\frac{(\gamma(n\vee m)-L)p_{\eta}}{1-p_{\eta}}\right)\:.
\end{eqnarray*}
Now by (\ref{eq:trivial bound}) we get $P(E_{2})=O(e^{-t})=O
(\epsilon)$
which gives
\[
\sigma_{4}=O\left(\gamma(n\vee m)e^{-\upsilon (n\wedge m)}+\epsilon\right)\:.
\]

Next, observe that $\tilde g(\bfY)$ has by Lemma \ref{lem3.1} the geometric
distribution with the parameter $p_\eta(p_\om+p_\eta-
p_\om p_\eta)^{-1}>p_\eta(p_\eta+p_\om)^{-1}$. Hence, in the same way as in
(\ref{3.29}) we obtain
\[
\sig_5\leq d_{TV}\big(\cL(\tilde g(\bfY)),\,\cL(\mbox{Geo}(\frac {p_\eta}
{p_\eta+p_\om}))\big)\leq 2p_\om\leq 2e^{-\upsilon n}.
\]

Combining all of our bounds we obtain,
\[
\sum_{j=1}^{5}\sigma_{j}=O\left(te^{-\upsilon\kappa}\gamma(n\vee m)+t\left(1+
\frac{p_{\omega}}{p_{\eta}}\right)(n\vee m)\cdot e^{-\upsilon\kappa}+t\psi(m)
(1+
\frac{p_{\omega}}{p_{\eta}})+\epsilon\right)\:.
\]
Recall that
\[
\epsilon=\max\{e^{-\upsilon\kappa},\psi(m)\}\text{ and }t=\epsilon^{-1/2},
\]
hence
\[
\sum_{j=1}^{5}\sigma_{j}=O\left(e^{-\upsilon\kappa/2}\gamma(n\vee m)+\left(1+
\frac{p_{\omega}}{p_{\eta}}\right)(n\vee m)\cdot e^{-\upsilon\kappa/2}+
\psi(m)^{1/2}
(1+\frac{p_{\omega}}{p_{\eta}})\right)\:.
\]
Now since $S$ from (\ref{eq:5 expressions}) is an arbitrary subset
of $\mathbb{N}$ the theorem follows.
\qed

\bibliography{matz_nonarticles,matz_articles}
\bibliographystyle{alpha}

\end{document}